\nonstopmode
\documentclass[reqno,10pt]{amsart}
\usepackage{latexsym}
\usepackage{fancyhdr}
\usepackage{amsmath, amssymb}
\usepackage[ansinew]{inputenc}
\usepackage[all]{xy}
\usepackage{pdflscape}
\usepackage{longtable}
\usepackage{rotating}
\usepackage{verbatim}
\usepackage{hyperref}
\usepackage{subfigure}
\usepackage{mathrsfs}
\usepackage{cleveref}
\usepackage{mdwlist}
\usepackage{color}

\newcounter{mnotecount}[section]

\newcommand{\rmnote}[1]{}%{\mnote{#1}}

\DeclareFontFamily{U}{mathb}{\hyphenchar\font45}
\DeclareFontShape{U}{mathb}{m}{n}{
      <5> <6> <7> <8> <9> <10> gen * mathb
      <10.95> mathb10 <12> <14.4> <17.28> <20.74> <24.88> mathb12
      }{}
\DeclareSymbolFont{mathb}{U}{mathb}{m}{n}
\DeclareFontSubstitution{U}{mathb}{m}{n}

\let\dot\relax
\DeclareMathAccent{\dot}{0}{mathb}{"39}
\let\ddot\relax
\DeclareMathAccent{\ddot}{0}{mathb}{"3A}
\let\dddot\relax
\DeclareMathAccent{\dddot}{0}{mathb}{"3B}
\let\ddddot\relax
\DeclareMathAccent{\ddddot}{0}{mathb}{"3C}

\theoremstyle{plain}
\newtheorem*{theorem*}{Theorem}
\newtheorem{theorem}{Theorem}[subsection]
\newtheorem*{lemma*}{Lemma}
\newtheorem{lemma}[theorem]{Lemma}
\newtheorem*{proposition*}{Proposition}
\newtheorem{proposition}[theorem]{Proposition}
\newtheorem*{corollary*}{Corollary}
\newtheorem{corollary}[theorem]{Corollary}
\newtheorem*{claim*}{Claim}

\newtheorem*{conjecture*}{Conjecture}

\newtheorem*{question*}{Question}
\theoremstyle{definition}
\newtheorem*{definition*}{Definition}
\newtheorem{definition}[theorem]{Definition}
\newtheorem*{example*}{Example}
\newtheorem{example}[theorem]{Example}

\newtheorem*{algorithm*}{Algorithm}
\newtheorem*{remark*}{Remark}
\newtheorem*{remarks*}{Remarks}
\newtheorem{remark}[theorem]{Remark}

\newtheorem*{convention*}{Convention}

\numberwithin{equation}{subsection}

\newcommand{\al}{\alpha}

\newcommand{\ep}{\epsilon}

\newcommand{\la}{\lambda}

\newcommand{\si}{\sigma}

\newcommand{\ta}{\tau}

\newcommand{\vh}{\varphi}

\newcommand{\C}{\mathbb{C}}

\newcommand{\I}{\mathbb{I}}

\newcommand{\R}{\mathbb{R}}

\newcommand{\p}{\partial}

\renewcommand{\Im}{\mathrm{Im}}
\renewcommand{\o}{\circ}

\newcommand{\on}{\operatorname}

\newcommand{\sr}[1]%
{\ifmmode{}^\dagger\else${}^\dagger$\fi\ifvmode
\vbox to 0pt{\vss
 \hbox to 0pt{\hskip\hsize\hskip1em
 \vbox{\hsize3cm\raggedright\pretolerance10000
 \noindent #1\hfill}\hss}\vss}\else
 \vadjust{\vbox to0pt{\vss%
 \hbox to 0pt{\hskip\hsize\hskip1em%
 \vbox{\hsize3cm\raggedright\pretolerance10000%
 \noindent #1\hfill}\hss}\vss}}\fi%
}

\newcommand{\ol}{\overline}

\title[]
{Roots of G\r{a}rding hyperbolic polynomials}

\author[A.~Rainer]{Armin Rainer}

\address{Fakult\"at f\"ur Mathematik, Universit\"at Wien, 
Oskar-Morgenstern-Platz~1, A-1090 Wien, Austria}
\email{armin.rainer@univie.ac.at}

\begin{document}

\begin{abstract}
We explore the regularity of the roots of G\r{a}rding hyperbolic polynomials and real stable polynomials. 
As an application we obtain new regularity results of Sobolev type for the eigenvalues of Hermitian matrices and for 
the singular values of arbitrary matrices. These results are optimal among all Sobolev spaces.   
\end{abstract}

\thanks{The author was supported by the Austrian Science Fund (FWF), START Programme Y963 and P~32905-N}
%\dedicatory{dedicatory}
\date{\today}

\maketitle

\subsection{Introduction} 

A celebrated theorem of Bronshtein \cite{Bronshtein79} (see also \cite{ParusinskiRainerHyp}) states that the roots of a monic univariate \emph{real-rooted} polynomial of degree $d$ with coefficients 
$C^{d-1,1}$-functions defined in some open set $U \subseteq \R^m$ can be represented by functions that are locally Lipschitz in $U$. 
(A $C^{k,\al}$-function is by definition a function that is $k$-times continuously differentiable and whose partial derivatives of order $k$ satisfy a local $\al$-H\"older condition.)
That means, 
if $a_j \in C^{d-1,1}(U)$, $j= 1,\ldots,d$, and all roots of the monic polynomial $$P_{a(x)}(Z) = Z^d + \sum_{j=1}^d a_j(x) Z^{d-j}$$
are real for all $x \in U$, then there exist functions $\la_j \in C^{0,1}(U)$, $j= 1,\ldots,d$, such that  
\begin{equation} \label{eq:Bronshtein}
	P_{a(x)}(Z) 
	= \prod_{j=1}^d (Z-\la_j(x)), \quad  x \in U.
\end{equation}
Let us call any $d$-tuple of functions $\la = (\la_1,\ldots,\la_d)$ that satisfy \eqref{eq:Bronshtein} 
a \emph{system of the roots} of $P_a$.
One can show that \emph{any} continuous system $\la = (\la_1,\ldots,\la_d)$ of the roots of $P_a$ with $a_j \in C^{d-1,1}(U)$, $j= 1,\ldots,d$, is locally Lipschitz.
And there are explicit bounds for the $C^{0,1}$-norm of the $\la_j$ in terms of the coefficients $a_j$; 
see \cite{ParusinskiRainerHyp}. 
We shall refer to these results as \emph{Bronshtein's theorem}.

That Bronshtein's theorem is sharp is manifest already in dimension $m=1$. 
If the coefficients $a=(a_1,\ldots,a_d)$ are of class $C^{d-1,1}$, 
then generally there are no differentiable systems of the roots of $P_a$; 
e.g.\ no function $\la$ with $\la^2=f$, where $f$ is the $C^{1,1}$-function $f(t) = t^2 \sin^2(\log |t|)$ for $t \ne0$ 
and $f(0)=0$, is differentiable at $0$.  
On the other hand, if $a$ is of class $C^d$ (resp.\ $C^{2d}$) and $m=1$, 
then there exists a $C^1$ (resp.\ twice differentiable) system of the roots of $P_a$; see \cite{ColombiniOrruPernazza12} (and also \cite{ParusinskiRainerHyp}).  
However, there are non-negative $C^\infty$-functions $f$ on $\R$ such that no function $\la$ with $\la^2 = f$ is of class 
$C^{1,\al}$ for any $\al>0$; cf.\ \cite{BBCP06}.

There are various ways to generalize univariate real-rooted polynomials to multivariate polynomials. In this paper we will 
consider \emph{G\r{a}rding hyperbolic polynomials} and \emph{real stable polynomials}. These classes of polynomials are intimately related 
and have important applications 
in many different fields including PDEs, combinatorics, optimization, functional analysis, probability etc.     

We will focus on the regularity of the roots of G\r{a}rding hyperbolic polynomials and real stable polynomials. 
By combining Bronshtein's theorem with the observation that the roots of G\r{a}rding hyperbolic polynomials are 
\emph{difference-convex} functions we obtain uniform regularity results for the roots of the considered classes of 
real-rooted polynomials  which in general are optimal among Sobolev spaces, see \Cref{thm:hyperbolic}.

As an application we obtain new regularity results for the eigenvalues of Hermitian matrices and for the singular values of arbitrary matrices.  
In particular, we will prove in \Cref{thm:Hermitian} that any $C^{1,1}$-curve of Hermitian $d \times d$ matrices admits a system of its eigenvalues 
that is locally of Sobolev class $W^{2,1}$. 
This result is uniform, in the sense explained in \Cref{thm:Hermitian}(3), and optimal among all Sobolev spaces $W^{k,p}$, see \Cref{rem:optimal}.   
(A $W^{k,p}$-function is by definition a $p$-integrable function with $p$-integrable weak partial derivatives up to order $k$.)

As a byproduct we obtain a new simple proof of the fact that a $C^1$-curve of Hermitian matrices admits a $C^1$-system of its eigenvalues. 
This was proved for symmetric matrices by \cite{Rellich69} and for normal matrices by \cite{RainerN} by different methods.  

\subsubsection*{Acknowledgment}
This note arose from a discussion with Alexander Lytchak who introduced me to difference-convex functions.

\subsection{Hyperbolic and real stable polynomials}
We call a univariate polynomial $f(Z) \in \C[Z]$ \emph{real-rooted} if all roots of $f$ are real; then also all coefficients are real 
and $f(Z) \in \R[Z]$. There is the following multivariate generalization.

\begin{definition}
	A polynomial $f(Z_1,\ldots,Z_n) \in \C[Z_1,\ldots,Z_n]$ is said to be \emph{stable} if 
	$f(z_1,\ldots,z_n) \ne 0$ for all $(z_1,\ldots,z_n) \in \C^n$ with $\Im (z_j) >0$ for $j = 1,\ldots,n$. 
	A stable polynomial with real coefficients is called a \emph{real stable} polynomial.  
\end{definition}

A univariate polynomial with real coefficients is real stable if and only if it is real-rooted. 
	A polynomial $f(Z_1,\ldots,Z_n) \in \C[Z_1,\ldots,Z_n]$ is stable (resp.\ real stable) 
	if and only if for all $x \in \R^n$ and all $v \in \R^n_{>0}$ the univariate polynomial $f(x + T v) \in \C[T]$ 
	is stable (resp.\ real stable).
Real stable polynomials played a crucial role in the recent proof of the Kadison--Singer conjecture \cite{Marcus:2015aa}.

A different but related multivariate generalization of real-rootedness is 
G\r{a}rding hyperbolicity.

\begin{definition}
	A homogeneous polynomial $f(Z_1,\ldots,Z_n) \in \R[Z_1,\ldots,Z_n]$ is said to be \emph{G\r{a}rding hyperbolic} 
	with respect to a given vector $v \in \R^n$ 
	if $f(v) \ne 0$ and 
	for all $x \in \R^n$ 
	the univariate polynomial $f(x - T v) \in \R[T]$ is real-rooted.
\end{definition} 

This notion was introduced by G\r{a}rding \cite{Garding51} in the 1950s. He showed that $f$ being hyperbolic with respect to a direction $v$ 
is a necessary and sufficient condition for local well-posedness of the Cauchy problem with principal symbol $f$ 
and initial data on a hyperplane with normal vector $v$. 
Hyperbolic polynomials have found many applications ever since, for instance, in Gurvits' proof \cite{Gurvits:2008aa} of the van der Waerden conjecture. 

\begin{example}
	The determinant on the real vector space of $d\times d$ Hermitian matrices 
	is G\r{a}rding hyperbolic with respect to the identity matrix $\I$.	
\end{example}

The relation between real stable and G\r{a}rding hyperbolic polynomials is captured in the following

\begin{proposition}[{\cite[Proposition 1.1]{Borcea:2010aa}}] \label{realstablehyperbolic}
	Let $f \in \R[Z_1,\ldots,Z_n]$ be of degree $d$ and let $f_H \in \R[Z_1,\ldots,Z_n, W]$ be the unique 
	homogeneous polynomial of degree $d$ such that 
	\[
		f_H(Z_1,\ldots,Z_n,1) = f(Z_1,\ldots,Z_n).
	\]
	Then $f$ is real stable if and only if $f_H$ is G\r{a}rding hyperbolic with respect to 
	every vector $v \in \R^{n}_{>0} \times \{0\}$.
\end{proposition}

Interesting examples of real stable polynomials (and hence of hyperbolic polynomials) 
are generated as follows (see \cite[Proposition 1.12]{Borcea:2010aa}):
Let $A_1,\ldots,A_n$ be positive semidefinite $m \times m$ matrices and $B$ a (complex) Hermitian $m \times m$ matrix. Then the 
polynomial
\[
		f(Z_1,\ldots,Z_n) = \det \Big( \sum_{j=1}^n Z_j A_j  + B\Big)
\]
is either real stable or identically zero. 
Conversely, any real stable polynomial in two variables $Z_1$ and $Z_2$ can be written as $\pm \det(Z_1 A_1 + Z_2 A_2 + B)$, 
where $A_j$ are positive semidefinite and $B$ is symmetric \cite[Theorem 1.13]{Borcea:2010aa}. The latter result is 
based on (and corresponds to) the Lax conjecture \cite{Lax:1958aa} for G\r{a}rding hyperbolic polynomials which was recently established by 
\cite{Lewis:2005aa} 
relying on work by \cite{Helton:2007aa}, \cite{Vinnikov:1993aa}; see also \Cref{rem:Lax}.   

We refer to the survey article \cite{BGLS01} for background on G\r{a}rding hyperbolic polynomials and 
more ways to generate examples.

\subsection{Difference-convex functions}

We shall work with the the class of difference-convex functions.
This is a subclass of the class of locally Lipschitz functions which arises as the smallest vector space that 
contains all continuous convex functions on a given set. We follow the survey article \cite{BacakBorwein11}; see also \cite{Hiriart-Urruty:1985aa}.

\begin{definition}
	Let $U \subseteq \R^n$ be a convex set.
A function $f : U \to \R$ is said to be \emph{difference-convex} ($DC$) if 
it can be written as the difference of two continuous convex functions on $U$.  	
A function $f$ is called \emph{locally difference-convex}  
if each point in the domain of $f$ has a convex neighborhood on which $f$ is 
difference-convex. 
\end{definition}

Let us denote by $DC(U)$ the space of all difference-convex functions on $U$    
and use the obvious notation $DC_{\on{loc}}(U)$ 
for the local class. We have $DC_{\on{loc}}(U) \subseteq C^{0,1}(U)$, by the local Lipschitz properties of locally bounded convex functions. 

Next we collect some basic properties of difference-convex functions.

\begin{lemma}[\cite{BacakBorwein11}] \label{properties}
Let $U \subseteq \R^n$ be an open convex set. Then:
\begin{enumerate}
  \item $DC(U)$ is an algebra. 
  \item Each $DC$-function on $U$ is the difference of two non-negative convex continuous functions on $U$.
  \item If $f \in DC(U)$ then positive and negative parts $f^\pm$ of $f$ as well as $|f|$  belong to $DC(U)$. 
\end{enumerate}
\end{lemma}

For us the following properties will be crucial.

\begin{lemma}[{\cite{Hartman:1959aa}}] \label{properties2}
 We have:
	\begin{enumerate}
		\item $f \in DC(\R^n)$ if and only if $f \in DC_{\on{loc}}(\R^n)$. 
  \item Let $A \subseteq \R^m$ be convex and either open or closed. 
  Let $B \subseteq \R^n$ be convex and open. If $f : A \to B$ is $DC$ (that is each component is $DC$) and 
  $g : B \to \R$ is $DC$, then $g\o f$ is locally $DC$.
  \item Let $A \subseteq \R^m$ be convex and either open or closed. 
  If a $DC$-function $f$ is nowhere vanishing on $A$, then $1/f$ is $DC$ on $A$.
	\end{enumerate}
\end{lemma}

We shall also need

\begin{lemma}[{\cite[Lemma 4.8]{VeselyZajicek89}}] \label{lem:mixing}
 	Let $U \subseteq \R^n$ be an open convex set and let 
 	$f_j \in DC(U)$, $j = 1,\ldots,d$.
 	If $f : U \to \R$ is continuous and $f(x) \in \{f_1(x),\ldots,f_d(x)\}$ for all $x \in U$, 
 	then $f \in DC(U)$.  
\end{lemma}

Note that, if $U \subseteq \R^n$ is open, then $C^{1,1}(U) \subseteq DC_{\on{loc}}(U)$, cf.\ \cite[Theorem 11]{Vesely:2003aa} 
and \cite[Section II]{Hiriart-Urruty:1985aa}. 
In particular, all real polynomials on $\R^n$ belong to $DC(\R^n)$; see also \cite[Section 3.3]{BacakBorwein11} for a direct argument.

On the other hand the first order partial derivatives of a difference-convex function $f : U \to \R$ 
have bounded variation, i.e., the weak second order partial derivatives of $f$ are signed Radon measures. 
This follows from Dudley's result \cite{Dudley:1977aa} that a Schwartz distribution is a convex function if and only if its second derivative 
is a non-negative matrix-valued Radon measure. 
In one dimension, a real-valued function $f$ on a compact interval is difference-convex if and only if 
$f$ is absolutely continuous and $f'$ has bounded variation.

\subsection{Roots of G\r{a}rding hyperbolic polynomials}

Let $f(Z_1,\ldots,Z_n) \in \R[Z_1,\ldots,Z_n]$ be a homogeneous polynomial of degree $d$ 
which is G\r{a}rding hyperbolic with respect to a direction $v \in \R^n$.
We may factorize
\[
	f(x + T v ) = f(v) \prod_{j=1}^d (T + \la_j^\downarrow(x)),
\]
where
\[
	\la^\downarrow_1(x) \ge \ldots \ge \la^\downarrow_d(x)
\]
are the decreasingly ordered roots of $f(x - T v) \in \R[T]$.  
We call 
\[
	\la^\downarrow = (\la^\downarrow_1, \ldots, \la^\downarrow_d) : \R^n \to \R^d	
\]
the \emph{characteristic map} of $f$ with respect to $v$ and $\la^\downarrow_1, \ldots, \la^\downarrow_d$ (in no particular order) the \emph{characteristic roots}.
The polynomial $f$ is recovered by 
$f(x) = f(v) \prod_{j=1}^d \la_j^\downarrow(x)$.	
We see that being G\r{a}rding hyperbolic with respect to a direction $v$ geometrically means that any 
affine line with direction $v$ meets the real hypersurface $\{x \in \R^n : f(x) = 0\}$ in $d$ points (with multiplicities).

Note that for all $j= 1,\ldots,d$, $r \in \R_{\ge 0}$, and $s \in \R$
we have  
\begin{align}
	\begin{split} \label{eq:homogeneous}
		\la^\downarrow_j(r x + s v) &= r \la^\downarrow_j(x) + s, \quad \text{ and }
	\\
	\la^\downarrow_j(-x) &= - \la^\downarrow_{d+1-j}(x).
	\end{split}   
\end{align}
Indeed, $f(rx +s v + T v) = r^d f(x + r^{-1}(s+T) v) = r^d f(v) \prod_{j=1}^d (r^{-1}(s+T) + \la_j^\downarrow(x)) = 
f(v) \prod_{j=1}^d (T + (s + r \la_j^\downarrow(x)))$, by homogeneity. The case $r=0$ is even simpler: $f(Tv) = T^d f(v)$.
The second assertion can be seen analogously; the negative factor 
reverses the order of the roots.

G\r{a}rding \cite[Theorem 2]{Garding59} proved that $\la^\downarrow_d$ is concave which, by \eqref{eq:homogeneous}, 
is equivalent to 
$\la^\downarrow_1$ being convex. 
In view of \eqref{eq:homogeneous} the largest root $\la^\downarrow_1$ is also positively homogeneous and thus sublinear.
Recall that a function $\vh$ on $\R^n$ is \emph{sublinear} 
if  $\vh(r x + s y) \le r\vh(x) + s \vh(y)$ for all $x,y\in \R^n$ and $r,s \ge 0$;  
it is easy to see that a positively homogeneous function $\vh$ is convex if and only if it is sublinear.    
The connected component $C_f$ of the set $\R^n \setminus \{f = 0\}$ which contains $v$ is an open convex cone, one has 
$C_f = \{x \in \R^n : \la^\downarrow_d(x) >0 \}$,  
and $f$ is G\r{a}rding hyperbolic with respect to each $w \in C_f$.

In \cite[Theorem 3.1]{BGLS01} a method for generating G\r{a}rding hyperbolic polynomials from a given one was presented.
Suppose that $f$ is a homogeneous polynomial of degree $d$ which is G\r{a}rding hyperbolic with respect to $v$ and has characteristic map $\lambda^\downarrow$.
If $g$ is a homogeneous symmetric polynomial of degree $e$ on $\R^d$ 
which is G\r{a}rding hyperbolic with respect to $e_1+e_2 + \cdots + e_d$ and has characteristic map $\mu^\downarrow$, 
then $g \o \la^\downarrow$ is G\r{a}rding hyperbolic of degree $e$ with respect to $v$ and its characteristic map is $\mu^\downarrow \o \la^\downarrow$.  
(Here and below $e_i$ denotes the $i$-th standard unit vector in $\R^d$.)

For instance, the homogeneous symmetric polynomial
\[
	g_k(Y_1,\ldots,Y_d) :=  \prod_{\substack{I \subseteq  \{1,\ldots,d\}\\|I| = k}}
	 \sum_{i \in I} Y_i.
\]
of degree $\binom{d}{k}$ is G\r{a}rding hyperbolic with respect to $e_1+e_2 + \cdots + e_d$ and has the characteristic roots 
$\mu_I(Y) = \frac{1}{k} \sum_{i \in I} Y_i$, 
where $I$ ranges over the subsets of $\{1,\ldots,d\}$ with $k$ elements. 
On the set $\{Y_1 \ge Y_2 \ge \cdots \ge Y_d\}$ 
we have $\mu_I(Y) \ge \mu_J(Y)$ if and only if $\sum_{i \in I}  i \le \sum_{j \in J} j$, in particular, 
$\mu_{\{1,\ldots,k\}}$ is the largest characteristic root. 
Then $g_k \o \la^\downarrow$ is G\r{a}rding hyperbolic with respect to $v$ and has 
the largest characteristic root $\mu_{\{1,\ldots,k\}} \o \la^\downarrow = \frac{1}{k} \sum_{i=1}^k \la_i^\downarrow$. 
From G\r{a}rding's result we may conclude that, for all $k = 1,\ldots, d$, the sum of the $k$ largest roots
\[
	\si_k := \sum_{i = 1}^k \la_{i}^\downarrow
\]
is a sublinear function on $\R^n$. We immediately obtain

\begin{proposition} \label{prop:DC}
	Let $f(Z_1,\ldots,Z_n) \in \R[Z_1,\ldots,Z_n]$ be a homogeneous polynomial of degree $d$ 
	which is G\r{a}rding hyperbolic with respect to a direction $v \in \R^n$.
	The characteristic map $\la^\downarrow : \R^n \to \R^d$ is globally Lipschitz and difference-convex on $\R^n$.
\end{proposition}

\begin{proof}
	Any finite sublinear function on $\R^n$ is globally Lipschitz; this follows from \cite[Corollary 10.5.1]{Rockafellar70}.
	By the above discussion, the functions $\si_k$, $k = 1,\ldots,d$, are convex and globally Lipschitz. 
	Thus each $\la^\downarrow_k = \si_k - \si_{k-1}$ is difference-convex and globally Lipschitz. 
\end{proof}

Combined with \Cref{lem:mixing} this leads to

\begin{corollary} \label{cor:DC}
	In the setting of \Cref{prop:DC} 
	every continuous function $\la : \R^n \to \R$
	satisfying $f(x - \la(x) v) = 0$ for all $x \in \R^n$ is globally Lipschitz and difference-convex on $\R^n$. 
\end{corollary}

\begin{proof}
	That $\la \in DC(\R^n)$ is an immediate consequence of \Cref{lem:mixing} and \Cref{prop:DC}.  
	For arbitrary $x, y \in \R^n$ consider the set $[x,y] := \{ (1-t) x + t y : t \in [0,1]\}$ equipped with the linear order induced by $[0,1]$. 
	By assumption, there is a finite sequence of points $x =: x_0  < x_1 < \cdots < x_k := y$ with $k \le d$ such that 
	for each $i = 1, \ldots, k$ there exists $j_i$ such that $\la(x_{i-1}) = \la_{j_i}^\downarrow(x_{i-1})$ and $\la(x_{i}) = \la_{j_i}^\downarrow(x_{i})$. 
	Then   
	\[
		\frac{|\la(y) - \la(x)|}{|y-x|} \le \sum_{i=1}^k \frac{|\la(x_i) - \la(x_{i-1})|}{|x_i - x_{i-1}|} \cdot \frac{|x_i - x_{i-1}|}{|y-x|}
		\le \sum_{i=1}^k \frac{|\la_{j_i}^\downarrow(x_i) - \la_{j_i}^\downarrow(x_{i-1})|}{|x_i - x_{i-1}|} 
	\]
	which implies that $\la$ is globally Lipschitz, by \Cref{prop:DC}. 
\end{proof}

Let $\on{Hyp}_n^d(v)$ denote the space of all homogeneous polynomials $f(Z_1,\ldots,Z_n) \in \R[Z_1,\ldots,Z_n]$ of degree $d$ 
which are G\r{a}rding hyperbolic with respect to a fixed vector $v \in \R^n$. 
By identifying $f\in \on{Hyp}_n^d(v)$ with the sequence of its coefficients (with respect to a fixed order of the monomials) 
we view $\on{Hyp}_n^d(v)$
as a subset of $\R^N$ where $N = \binom{n+d-1}{d}$. 
Then the subset $\mathring{\on{Hyp}}_n^d(v)$, consisting of all those $f \in \on{Hyp}_n^d(v)$ for which the roots of $f(x - T v)$ 
are simple for every $x$ not proportional to $v$, is open.
And $\on{Hyp}_n^d(v)$ is the part of the closure of $\mathring{\on{Hyp}}_n^d(v)$ where $f(v) \ne 0$.  
See \cite{Nuij68} in which also the connectivity properties of these spaces are discussed.

By a \emph{$C^{k,\al}$-mapping} $f : \R^m \to \on{Hyp}_n^d(v)$ we mean a mapping with values in $\on{Hyp}_n^d(v)$ which is of class $C^{k,\al}$ 
when considered as mapping to $\R^N$.   
For each $y \in \R^m$ we have the characteristic map $\la^\downarrow(y)(\cdot)$ of $f(y)$. By slight abuse of notation we write
\[
	\la^\downarrow : \R^m \times \R^n \to \R^d
\] 
and call it the \emph{characteristic map} of $f$ with respect to $v$. 
Its regularity properties are stated in the following

\begin{proposition} \label{prop:Lip}
	If $f : \R^m \to \on{Hyp}_n^d(v)$ is a mapping of class $C^{d-1,1}$, 
	then its characteristic map $\la^\downarrow : \R^m \times \R^n \to \R^d$ is locally Lipschitz. 	
\end{proposition} 

\begin{proof}
	Bronshtein's theorem \cite{Bronshtein79} (see also \cite{ParusinskiRainerHyp}) implies that the 
	roots $\la^\downarrow_j(y,x)$, $j =1,\ldots,d$ of the univariate 
	real-rooted polynomial 
	\begin{equation} \label{eq:real-rooted}
		\R^m \times \R^n \ni (y,x) \mapsto f(y)(x - Tv) \in \R[T]
	\end{equation}
	are locally Lipschitz in $(y,x)$.
\end{proof}

Note that in \Cref{prop:Lip} we need the assumption that the coefficients of $f$ are of class $C^{d-1,1}$, since the assumptions 
of Bronshtein's theorem are sharp.

By \Cref{prop:DC} the behavior of the characteristic map 
$\la^\downarrow : \R^m \times \R^n \to \R^d$, $(y,x) \mapsto \la^\downarrow(y,x)$ 
is essentially better in the $x$ variables. 
This will be further investigated in the next theorem, that is \Cref{thm:hyperbolic}. 
To this end we need the following

\begin{lemma} \label{Luzin}	
  Let $\la = (\la_1,\ldots,\la_d) : \R \to \R^d$ be of class $C^1$ and assume that the coefficients of the polynomial 
  $P(t)(Z) = \prod_{j=1}^d (Z-\la_j(t))$ are of class $W^{2,1}_{\on{loc}}$. 
  Then each derivative $\la_j'$, $j=1,\ldots,d$, has the Luzin (N) property  
  (i.e.\ sets of Lebesgue measure zero are mapped to sets of Lebesgue measure zero).
\end{lemma}

\begin{proof}
  Note that it is enough to show the Luzin (N) property on bounded intervals, since $\R$ can be covered 
  by countably many of them. 	
  We proceed by induction on $d$. The case $d=1$ is clear, since in this case the derivative of the only root 
  is locally absolutely continuous, by assumption. Let $d>1$.

  Let $A$ denote the closed set of all $t \in \R$ such that $\la_1(t) = \cdots = \la_d(t)$. 
  On an open neighborhood $I$ of each point $t_0 \in A^c := \R \setminus A$ we have a splitting 
  $P(t) = P_1(t) P_2(t)$ into polynomials with positive degree 
  and $W^{2,1}_{\on{loc}}$ coefficients, 
  by the inverse function theorem (cf.\ \cite[Lemma 3.2]{ParusinskiRainerHyp});  
  here we use that the superposition with a real analytic function preserves the class $W^{2,1}_{\on{loc}}$, cf.\ 
  \cite[Theorem 2]{Bourdaud:1991aa}. 
  On $I$ we have a partition of $\{\la_1,\ldots,\la_d\}$ into the roots of $P_1$ and the roots of $P_2$.
  By induction hypothesis, 
  for each $j$ the restriction $\la_j'|_I$ has the Luzin (N) property.
  Since $A^c$ can be covered by countably many open neighborhoods of type $I$, we 
  see that each $\la_j'|_{A^c}$ has the Luzin (N) property. 

  Let $A'$ denote the set of accumulation points of $A$. 
  Let $a_1(t)$ denote the coefficient of $Z^{d-1}$ in $P(t)$.
  For $t \in A$ we have 
  \[
  	a_1(t) = - \la_1(t) - \cdots - \la_d(t) =  - d\cdot  \la_j(t), \quad \text{ for all } j. 
  \]
  And hence for $t \in A'$ and all $j$ we have 
  \[
	\la_j'(t) = \lim_{A \ni t_n \to t} \frac{\la_j(t_n) - \la_j(t)}{t_n - t} =  
	- \frac{1}{d}\lim_{A \ni t_n \to t} \frac{a_1(t_n) - a_1(t)}{t_n - t} = - \frac{a_1'(t)}{d}. 	
  \]
  By assumption $a_1'$ is absolutely continuous and hence has the Luzin (N) property. 
  If $E \subseteq \R$ is a set of Lebesgue measure zero, then also
  \begin{align*}
    \la_j'(E)  
    &= \la_j'(E\cap A^c) \cup \la_j'(E \cap A') \cup \la_j'(E \cap A \setminus A') 
    \\
    &= \la_j'(E\cap A^c) \cup  (-\tfrac {1} d) a_1'(E \cap A') \cup \la_j'(E \cap A \setminus A')
  \end{align*}
  has Lebesgue measure zero, noting that $A \setminus A'$ is discrete and thus countable.  
\end{proof}

Let $f \in \on{Hyp}_n^d(v)$ have characteristic map $\lambda^\downarrow : \R^n \to \R^d$. 
Let $x : \R \to \R^n$ be a curve. 
A mapping $\la = (\la_1,\ldots,\la_d) : \R \to \R^d$ 
is called a \emph{system of the roots of $f$ along $x$} if 
$\la(t)$ and $\la^\downarrow(x(t))$ coincide as unordered $d$-tuples for all $t$. 

\begin{lemma} \label{lem:differentiable}
	Let $f \in \on{Hyp}^d_n(v)$
	and let $x:\R \to \R^n$ be a $C^1$-curve.
	Then there exists a differentiable system $\la = (\la_1,\ldots,\la_d)$ of the roots of $f$ along $x$. 	
\end{lemma}

\begin{proof}
	Let $\la_0$ be a $p$-fold root of $f(x(t_0)-Tv)$. 
	We claim that the polynomial $f(x(t)-Tv)$ has $p$ roots of the form $\la_0 + (t-t_0) \mu_j(t)$, where the $\mu_j$ are functions which are 
	defined near $t_0$ and continuous at $t_0$. 

	We may assume without loss of generality that $t_0=0$. 
	In view of \eqref{eq:homogeneous} we may also assume that $\la_0 =0$. 
	We consider the localization $f_{x(0)}$ of $f$ at $x(0)$ defined by 
	\begin{equation} \label{eq:localization}
		t^d f(t^{-1} x(0) + \xi) = f(x(0) + t \xi) = t^p f_{x(0)}(\xi) + O(t^{p+1}),
	\end{equation}
	where $f_{x(0)}(\xi)$ is the first coefficient in the above expansion which does not vanish identically. 
	Since 
	\[
		f(x(0) + t v) = f(v) \prod_{j=1}^d (t + \la^\downarrow_j(x(0))) =  t^p f(v) \prod_{\la^\downarrow_j(x(0)) \ne \la_0} (t + \la^\downarrow_j(x(0))),  
	\]
	the exponent $p$ in \eqref{eq:localization} is precisely the multiplicity of $\la_0$.  
	The homogeneous polynomial $f_{x(0)}$ has degree $p$ and is G\r{a}rding hyperbolic with respect to $v$; see \cite[Lemma 3.42]{ABG70}.

	Since $x$ is of class $C^1$, we have $x(t) = x(0) + t \tilde x(t)$ for a continuous function $\tilde x$. We are interested in the solutions of the equation
	\[
		f(x(0) + t \tilde x(t) - t Y v) = 0
	\]
	By \eqref{eq:localization} we may equivalently study the equation	
	\[
		f_{x(0)}(\tilde x(t) - Y v) + O(t) = 0.
	\] 
	This is a polynomial equation in the unknown $Y$ with coefficients which depend continuously on $t$ in a neighborhood of $0$.  
	It follows that there exist $p$ solutions $\mu_1(t), \ldots,\mu_p(t)$ (since $p$ is the degree of $f_{x(0)}$) for $t$ near $0$ 
	which are continuous at $t=0$. This implies the claim. 

	Let us indicate how to construct a global differentiable system $\la = (\la_1,\ldots,\la_d)$ 
	of the roots of $f$ along $x$ from the local data 
	provided by the claim; we follow the argument in \cite[4.3]{AKLM98}.	
	We start with the continuous system $\nu_1 \ge \nu_2 \ge \cdots \ge \nu_d$, where $\nu_j := \la^\downarrow_j \o x$. 
	Fix $i \ne j$ and $t_0 \in \R$. 
	If $\nu_i(t_0) = \nu_j(t_0)$, then, by the claim, 
	$\nu_i(t) - \nu_j(t) = (t-t_0) \nu_{ij}(t)$ near $t_0$, where $\nu_{ij}$ is a function which is continuous at $t_0$.
	Let $E_{ij}$ be the closed set of $t_0 \in \R$ such that $\nu_i(t_0) = \nu_j(t_0)$ and $\nu_{ij}(t_0) =0$.
	At the points of $E_{ij}$ the roots $\nu_i$ and $\nu_j$ are differentiable and the derivatives $\nu_i'$ and $\nu_j'$ coincide. 
	  % $E_{ij}$ closed since $\nu_{ij}(t;t_0) = (\nu_{i}(t) -\nu_j(t))$ is continuous in (t,t_0)
	Let us define 
	\[
		\la_k(t) := \nu_{\si_t(k)}(t), \quad k = 1,\ldots,d,  
	\]  
	where $\si_t$ is the permutation
	\[
		\si_t = (1\; 2)^{\ep_{12}(t)} \cdots (1\; d)^{\ep_{1d}(t)} (2\; 3)^{\ep_{23}(t)} \cdots (2\; d)^{\ep_{2d}(t)} \cdots (d-1\; d)^{\ep_{d-1,d}(t)}
	\]
	and $\ep_{ij}(t) \in \{0,1\}$ are determined as follows. 
	In each connected component $I$ of $\R \setminus E_{ij}$ choose a point $t_I$ and set $\ep_{ij}(t_I) =0$. 
	Going left and right from $t_I$ change $\ep_{ij}$ in each point $t_0 \in I$, where $\nu_i(t_0) = \nu_j(t_0)$ and $\nu_{ij}(t_0) \ne 0$ 
	(including $t_I$ if it is such a point). 
	These points accumulate only in $E_{ij}$, where any choice for $\ep_{ij}$ is good. 
	It follows that $\la = (\la_1,\ldots,\la_d)$ is a global differentiable system of the roots of $f$ along $x$.  
\end{proof}

We are ready to state and prove the theorem.

\begin{theorem} \label{thm:hyperbolic}
	Let $f \in \on{Hyp}_n^d(v)$. Let $x \in DC(\R, \R^n) \cap W^{2,1}_{\on{loc}}(\R, \R^n)$; 
	this is, for instance, fulfilled if $x \in C^{1,1}(\R,\R^n)$.  
	\begin{enumerate}
		\item Any differentiable system $\la = (\la_1,\ldots,\la_d)$ of the roots of $f$ along $x$ is actually of class 
		\begin{equation} \label{eq:hyperbolic}
			\la \in C^1(\R,\R^d) \cap DC(\R,\R^d)\cap  W^{2,1}_{\on{loc}}(\R,\R^d).	
		\end{equation}
		\item If $x : \R \to \R^n$ is of class $C^1$, 
		then there exists a differentiable (thus $C^1$) system $\la = (\la_1,\ldots,\la_d)$ of the roots of $f$ along $x$. 
		\item The result is uniform in the following sense:
		Let $I \subseteq \R$ be a bounded open interval.
		Let $U$ be an open neighborhood of  the closure of $I^{1+k}$ in $\R^{1+k}$.
		Suppose that $x : U \to \R^n$ is such that 
		\begin{itemize}
			\item $x$ is locally DC on $U$,
			\item $x( \cdot,r) \in C^1(\ol I,\R^n) \cap W^{2,1}(I,\R^n)$ for all $r \in \ol I^k$.
		\end{itemize}
		Assume that, for each $r \in \ol I^k$, a $C^1$-system $\la(\cdot,r)$ of the roots of $f$ along $x(\cdot,r)$ is fixed. 
		Then the family 
		$\la(\cdot,r)$, for $r \in \ol I^k$, is bounded in $C^1(\ol I, \R^d)$ and there is a non-negative 
		$L^1$-function $m : I^k \to \R_{\ge 0}$ such that  
		\begin{equation} \label{eq:uniform}
			\|\la(\cdot, r)\|_{W^{2,1}(I,\R^d)} \le m(r), \quad \text{ for a.e.\ } r \in I^k.
		\end{equation}
	\end{enumerate} 
\end{theorem}

\begin{proof}
	(1) Let $\la = (\la_1,\ldots,\la_d)$ be a differentiable system of the roots of $f$ along $x$. 
	By \Cref{properties2}, $\la^\downarrow \o x$ is $DC$ on $\R$.
	Then $\la$ is $DC$ on $\R$, by \Cref{cor:DC}. 
	A differentiable $DC$-function $\vh$ on $\R$ is of class $C^1$ and the difference of two convex $C^1$-functions (see \cite[Theorem 2.1]{Hiriart-Urruty:1985aa}); 
	indeed, by Darboux's theorem, $\vh'$ has the intermediate value property and so $\vh'$ having (locally) bounded variation implies that $\vh'$ is continuous. 
	It follows that $\la$ is of class $C^1$.

	In order to conclude that $\la  \in W^{2,1}_{\on{loc}}(\R,\R^n)$ it suffices 
 	to show that each $\la_j'$ has the Luzin (N) property which follows from Lemma \ref{Luzin}; 
 	in fact a continuous function of pointwise bounded variation with the Luzin (N) property is absolutely continuous, cf.\ 
 	\cite[Corollary 3.27]{Leoni09}.

 	(2) This is \Cref{lem:differentiable}.

	(3)	The Lipschitz constant of $(\lambda^\downarrow \o x)( \cdot,r)$ on $\ol I$ is uniformly bounded in $r \in \ol I^k$ by the chain rule, 
	 	since $\lambda^\downarrow$ is Lipschitz, by \Cref{prop:DC}. By an argument similar to the one in the proof of \Cref{cor:DC} 
	 	we may infer the boundedness of the Lipschitz constant of $\la(\cdot,r)$ on $\ol I$. 
	 	It follows that the family $\la(\cdot,r)$, $r \in \ol I^k$, is bounded in $C^1(\ol I,\R^d)$.

		By \Cref{properties2}, $\la^\downarrow \o x$ is locally $DC$ on $U$. 
 		In particular, the directional variation $V_{e_1}(\p_{1}(\la^\downarrow \o x), I^{1+k})$ of $\p_{1} (\la^\downarrow \o x)$ on $I^{1+k}$ is finite, where $e_1$ denotes 
 		the first standard unit vector. We have, see e.g \cite[(3.100)]{AFP00},
 		\[
 			V_{e_1}(\p_{1}(\la^\downarrow \o x), I^{1+k}) = \int_{I^k} V(\p_1 (\la^\downarrow \o x)(\cdot,r), I)\, dr,
 		\]
 		where $V(\p_1 (\la^\downarrow \o x)(\cdot,r), I)$ is the variation of $\p_1 (\la^\downarrow \o x)(\cdot,r)$ on $I$. 
 		The variation of $\p_1 (\la^\downarrow \o x)(\cdot,r)$ on $I$ dominates the variation of $\p_1 \la(\cdot,r)$ on $I$ (up to a constant which depends only on $d$), since 
 		for each $j = 1,\ldots, d$ and a.e.\ $t$ we have 
 		$$\p_1 \la_j(t,r) \in \{\p_1 (\la_1^\downarrow \o x)(t,r),\ldots, \p_1 (\la_d^\downarrow \o x)(t,r)\},$$ 
 		and $\p_1 \la_j(t,r)$ is continuous in $t$. 
 		Since $\p_1 \la_j(\cdot,r)$ is absolutely continuous on $I$, by (1), its $W^{1,1}$-norm coincides with its $BV$-norm on $I$. 
 		Together with the boundedness in $C^1(\ol I,\R^d)$ the assertion \eqref{eq:uniform} follows easily.
\end{proof}

\begin{remark} 
	If the system $\la = (\la_1,\ldots,\la_d)$ of the roots of $f$ along $x$ is not differentiable at $t_0$, then 
		$\la$ is not of class $W^{2,1}$ in any neighborhood of $t_0$. 
\end{remark}

In several variables the characteristic roots do not admit differentiable rearrangements.  
But for real analytic families of G\r{a}rding hyperbolic polynomials the characteristic roots can be chosen 
real analytically locally after blowing up the parameter space.

\begin{proposition} \label{prop:RA}
		If $f : \R^m \to \on{Hyp}_n^d(v)$ is real analytic and $x : \R^\ell \to \R^n$ is real analytic, 
		then there exists a locally finite composite of blowing-ups with smooth centers $\ta=(\ta_1,\ta_2) : Y \to \R^m \times \R^\ell$ 
		such that 
		for each $y_0 \in Y$ there is a neighborhood $U$ and a real analytic system of the roots 
		of the univariate polynomial $$U \ni y \mapsto f(\ta_1(y))(x(\ta_2(y)) - T v) \in \R[T].$$	
		If $f : \R \to \on{Hyp}_n^d(v)$ is real analytic and $x : \R \to \R^n$ is a real analytic curve, 
		then there exists a global real analytic system of the roots of 
		$$\R \ni s \mapsto f(s)(x(s) - T v) \in \R[T].$$
\end{proposition}

\begin{proof}
	The statement follows from \cite[Theorem 5.8]{KurdykaPaunescu08} applied to the real-rooted polynomial 
	\eqref{eq:real-rooted}. The (one-dimensional) supplement follows from Rellich's theorem \cite{Rellich37I}; see also \cite[Theorem 5.1]{AKLM98}.
\end{proof}

\begin{remark}
	Analogous results follow for suitable quasianalytic classes thanks to a corresponding theorem for 
	real-rooted polynomials, see \cite{RainerQA}. 
\end{remark}

\begin{remark}
	It is shown in \cite[Theorem 4.2]{Harvey:2013aa} that for any $f \in \on{Hyp}_n^d(v)$, $x_0 \in \R^n$, and $w \in C_f$ 
	there is a system $\la= (\la_1,\ldots,\la_d) : \R \to \R^d$ of the roots of $f$ along $t \mapsto x_0 + t w$ which is 
	real analytic and such that each $\la_j : \R \to \R$ is strictly increasing and surjective. The inverses of the $\la_j$ 
	form a system of the characteristic roots of $f  \in \on{Hyp}_n^d(w)$ along $t \mapsto x_0 + t v$. 
\end{remark}

\subsection{Roots of real stable polynomials}

The regularity results for G\r{a}rding hyperbolic polynomials of the previous section immediately give regularity results 
for real stable polynomials thanks to the connection presented in \Cref{realstablehyperbolic}.
Let us denote by $\on{RStab}^d_n$ the set of real stable polynomials of degree $d$ in $n$ variables.
Then \Cref{realstablehyperbolic} defines a bijection 
\[
	\on{RStab}^d_n \to \bigcap_{v \in \R^n_{>0} \times \{0\}} \on{Hyp}^d_n(v), \quad f \mapsto f_H.
\]	
Suppose that $f \in \on{RStab}^d_n$. For all $v \in \R^n_{>0} \times \{0\}$ 
we may consider the characteristic map $\la^\downarrow_v : \R^{n+1} \to \R^d$ of $f_H$ with respect to $v$.  
Then the mapping $\la^\downarrow_v|_{\{x_{n+1}=1\}}$ is globally Lipschitz and difference-convex and its 
components represent the roots of $f(x' - T v') \in \R[T]$, where $x' = (x_1,\ldots,x_n)$ and $v'=(v_1,\ldots,v_n)$. 
It is then straightforward to transfer the regularity results for G\r{a}rding hyperbolic polynomials of  
the previous section.

Finally, note that we can also vary the direction $v$: 
if $f  : \R^m \to \on{RStab}^d_n$ \emph{and}
$v : \R^\ell \to \R^n_{>0}$ are $C^{d-1,1}$-mappings, 
then the roots of the univariate real-rooted polynomial 
\[
	\R^m \times \R^\ell \times \R^n \ni (s,t,x) \mapsto f(s)(x - T v(t)) \in \R[T]	
\]
are locally Lipschitz in $(s,t,x)$, by Bronshtein's theorem \cite{Bronshtein79} (see also \cite{ParusinskiRainerHyp}).

\subsection{Eigenvalues of Hermitian matrices}

Let $\on{Herm}(d)$ denote the real vector space of complex $d \times d$ Hermitian matrices. 
The determinant $\det$ is G\r{a}rding hyperbolic with respect to the identity matrix $\mathbb I \in \on{Herm}(d)$. 
The characteristic map $\la^\downarrow$ of $\det$ with respect to $\I$ assigns to a Hermitian matrix $A \in \on{Herm}(d)$ its eigenvalues in 
decreasing order
\[
	\la^\downarrow_1(A) \ge \ldots \ge \la^\downarrow_d(A).
\] 
As a direct consequence of \Cref{prop:DC} and the discussion preceding it we get

\begin{corollary} \label{cor:HermLip}
	The characteristic map 
	$\la^\downarrow : \on{Herm}(d)  \to \R^d$ 
	is globally Lipschitz and difference-convex on $\on{Herm}(d)$. 
	The sum $\sum_{i=1}^k \la^\downarrow_i$, for $k = 1,\ldots,d$, of the $k$ largest eigenvalues is sublinear. 
\end{corollary}

The following result is a special case of
\Cref{thm:hyperbolic}. 

\begin{theorem} \label{thm:Hermitian}
    Let $A : \R \to \on{Herm}(d)$ be a curve of $d \times d$ Hermitian matrices 
    of class $DC \cap W^{2,1}_{\on{loc}}$ on $\R$.
    \begin{enumerate}
    	\item Any differentiable system $\la=(\la_1,\ldots,\la_d)$ of the eigenvalues of $A$ is actually of class 
    	\[
				\la \in C^1(\R,\R^d) \cap DC(\R,\R^d)\cap  W^{2,1}_{\on{loc}}(\R,\R^d).    		
    	\]  
    	\item If $A$ additionally is of class $C^1$, then there exists a differentiable (thus $C^1$) system  
    	$\la=(\la_1,\ldots,\la_d)$ of the eigenvalues of $A$.
    	\item  The result is uniform in the following sense:
		Let $I \subseteq \R$ be a bounded open interval.
		Let $U$ be an open neighborhood of  the closure of $I^{1+k}$ in $\R^{1+k}$.
		Suppose that $A : U \to \on{Herm}(d)$ is such that 
		\begin{itemize}
			\item $A$ is locally DC on $U$,
			\item $A( \cdot,r)$ is of class $C^1 \cap W^{2,1}$ on $\ol I$ for all $r \in \ol I^k$.
		\end{itemize}
		Assume that, for each $r \in \ol I^k$, a $C^1$-system $\la(\cdot,r)$ of the eigenvalues of $A(\cdot,r)$ is fixed. 
		Then the family 
		$\la(\cdot,r)$, for $r \in \ol I^k$, is bounded in $C^1(\ol I, \R^d)$ and there is a non-negative 
		$L^1$-function $m : I^k \to \R_{\ge 0}$ such that  
		\begin{equation*} 
			\|\la(\cdot, r)\|_{W^{2,1}(I,\R^d)} \le m(r), \quad \text{ for a.e.\ } r \in I^k.
		\end{equation*}
    \end{enumerate}  
\end{theorem}

Note that the assumptions on $A$ are in particular satisfied if $A$ is of class $C^{1,1}$.

\begin{remark} \label{rem:optimal}
  The conclusion of this theorem is best-possible among all Sobolev spaces $W^{k,p}$. Indeed, by the Sobolev inequality,
  $W^{k,p}$-regularity with $k+p>2$ would imply $C^{1,\al}$-regularity with some $\al>0$, contradicting the counter-example 
  of \cite{KM03}: 
  there is a $C^\infty$-curve of symmetric $2 \times 2$ matrices the eigenvalues of which do not admit a $C^{1,\al}$-system 
  for any $\al>0$. 
\end{remark}

\begin{remark} \label{rem:Lax}
Statement (2) in \Cref{thm:Hermitian} is due to Rellich \cite{Rellich69} in the case of symmetric matrices, and it was proved 
for normal matrices in \cite{RainerN}. \Cref{lem:differentiable} gives an independent proof for Hermitian matrices.

Conversely, (2) in \Cref{thm:Hermitian} implies \Cref{lem:differentiable} if $n\le 3$: 
The Lax conjecture (posed 1958 in \cite{Lax:1958aa}) which by now is a theorem (as \cite{Lewis:2005aa} discovered the conjecture follows from a theorem 
of \cite{Helton:2007aa}, \cite{Vinnikov:1993aa}) states
that a homogeneous polynomial $f$ on $\R^3$ is G\r{a}rding hyperbolic of degree $d$ with respect to the direction $e_1$ with $f(e_1) = 1$ 
if and only if there exist real symmetric $d \times d$ matrices $A$ and $B$ such that 
\begin{equation} \label{eq:Lax}
	f(x,y,z) = \det ( x \mathbb I + y A + z B).
\end{equation}
So, for $n\le 3$, \Cref{thm:hyperbolic} follows from \Cref{thm:Hermitian}. 
A representation of type \eqref{eq:Lax} is in general not possible for G\r{a}rding hyperbolic polynomials on $\R^n$ with $n \ge 4$. 
Indeed, the dimension of $\on{Hyp}^d_n(v)$ is $\binom{n+d -1}{d}$ while the set of polynomials in $\R[X_1,\ldots,X_n]$ of the form 
\begin{equation} \label{eq:det}
	\det(X_1 \I + X_2 A_2 + \cdots + X_n A_n),	
\end{equation}
where $A_i$ are real symmetric $d \times d$ matrices, is at most $(n-1) \cdot \binom{d+1}{2}$.

A particular homogeneous polynomial of degree $2$ which is G\r{a}rding hyperbolic with respect to $e_1$ but cannot be represented in the form 
	\eqref{eq:det} 
	is the Lorentzian polynomial $f(X) =X_1^2 - X_2^2 - \cdots - X_n^2$ for $n \ge 4$. Indeed, if $x_1 =0$ and $x_j \ne 0$ for some $j\in {2,\ldots,n}$ 
	then $f(x) <0$. But for any choice of real symmetric matrices $A_i$ we may find $x \in \R^n \setminus \{0\}$ with $x_1 =0$ such that 
	the first row of $x_1 \I + x_2 A_2 + \cdots + x_n A_n$ is zero. 
Cf.  \cite[p.2498]{Lewis:2005aa}.
\end{remark}

\subsection{Singular values}
Let $A$ be any $m \times n$ matrix with complex entries and let 
$$\si^\downarrow_1(A) \ge \si^\downarrow_2(A) \ge \cdots \ge \si^\downarrow_n(A) \ge 0$$ 
be the singular values of $A$ in decreasing order, i.e., the non-negative square roots of the eigenvalues of $A^* A$. 
If $\on{rank} A = \ell$, then $\si^\downarrow_{\ell+1}(A) = \cdots = \si^\downarrow_{n}(A) = 0$.
Thus we set $\ell := \min\{m,n\}$ and consider only $\si^\downarrow_{j} (A)$, for $j = 1,\ldots,\ell$.

We may consider the $\si^\downarrow_j$ as functions on the vector space $\C^{m \times n}$ of complex $m \times n$ matrices.
Without loss of generality assume that $m \le n$ and let $\tilde A$ be the $n \times n$ matrix resulting from $A$ by 
	adding $n-m$ rows consisting of zeros. Then the eigenvalues of the Hermitian matrix 
	\begin{equation} \label{eq:Hermitianext}
		\begin{pmatrix}
			0 & \tilde A \\
			\tilde A^* & 0
		\end{pmatrix}
	\end{equation}
	are precisely 
	\[
		\si^\downarrow_1(A) \ge   \cdots \ge \si^\downarrow_n(A) \ge 
		-\si^\downarrow_n(A) \ge  \cdots \ge - \si^\downarrow_1(A).
	\]
It follows from \Cref{cor:HermLip} that, for all $k = 1,\ldots,\ell$, the sum 
\begin{equation} \label{eq:KyFan}
	\sum_{j = 1}^k \si^\downarrow_j
\end{equation}
of the $k$ largest singular values is a sublinear function on $\C^{m \times n}$ viewed as a real vector space. In fact, the sums 
$A \mapsto \sum_{j = 1}^k \si^\downarrow_j(A)$ are the so-called \emph{Ky Fan norms}. 
We obtain

\begin{corollary} 
	The mapping 
	$\si^\downarrow = (\si^\downarrow_1, \ldots,\si^\downarrow_\ell) : \C^{m \times n}  \to \R^\ell$, where $\ell = \min\{m,n\}$, 
	is globally Lipschitz and difference-convex on $\C^{m \times n}$.
\end{corollary}

The following theorem is a consequence of \Cref{thm:Hermitian}.

\begin{theorem}
    Let $A : \R \to \C^{m \times n}$ be a curve of $m \times n$ complex matrices
    of class $C^1 \cap DC \cap W^{2,1}_{\on{loc}}$.  
    If $\on{rank} A(t) = \min\{m,n\}=:\ell$ for all $t$, then  
    there exists a system $\si = (\si_1,\ldots,\si_\ell)$ of the singular values of $A$ 
    such that $\si \in C^1(\R,\R^\ell) \cap DC(\R,\R^\ell)\cap W^{2,1}_{\on{loc}}(\R,\R^\ell)$.
    This result is uniform in the sense explained in \Cref{thm:Hermitian}.
\end{theorem}

\begin{proof}
	The assertions follow from \Cref{thm:Hermitian} applied to the Hermitian matrix \eqref{eq:Hermitianext}. 
	The condition on the rank of $A$ guarantees that 
	the non-trivial singular values of $A$ are always strictly positive and hence there exists a $C^1$-system of them, 
	since there exists a $C^1$-system of the eigenvalues of \eqref{eq:Hermitianext}.  
\end{proof}

The rank condition is necessary; for instance, the singular value of the symmetric $1\times 1$ matrix $A(t) = (t)$ is $|t|$
which does not admit a $C^1$ parameterization.

%\bibliography{../../../references/biblio}

\begin{thebibliography}{10}

\bibitem{AKLM98}
D.~Alekseevsky, A.~Kriegl, M.~Losik, and P.~W. Michor, \emph{Choosing roots of
  polynomials smoothly}, Israel J. Math. \textbf{105} (1998), 203--233.

\bibitem{AFP00}
L.~Ambrosio, N.~Fusco, and D.~Pallara, \emph{Functions of bounded variation and
  free discontinuity problems}, Oxford Mathematical Monographs, The Clarendon
  Press Oxford University Press, New York, 2000.

\bibitem{ABG70}
M.~F. Atiyah, R.~Bott, and L.~G{\.a}rding, \emph{Lacunas for hyperbolic
  differential operators with constant coefficients. {I}}, Acta Math.
  \textbf{124} (1970), 109--189.

\bibitem{BacakBorwein11}
M.~Ba{\v{c}}{{\'a}}k and J.~M. Borwein, \emph{On difference convexity of
  locally {L}ipschitz functions}, Optimization \textbf{60} (2011), no.~8-9,
  961--978. 

\bibitem{BGLS01}
H.~H. Bauschke, O.~G{\"u}ler, A.~S. Lewis, and H.~S. Sendov, \emph{Hyperbolic
  polynomials and convex analysis}, Canad. J. Math. \textbf{53} (2001), no.~3,
  470--488.

\bibitem{BBCP06}
J.-M. Bony, F.~Broglia, F.~Colombini, and L.~Pernazza, \emph{Nonnegative
  functions as squares or sums of squares}, J. Funct. Anal. \textbf{232}
  (2006), no.~1, 137--147.  

\bibitem{Borcea:2010aa}
J.~Borcea and P.~Br{\"a}nd{\'e}n, \emph{{Multivariate P{\'o}lya-Schur
  classification problems in the Weyl algebra}}, Proceedings of the London
  Mathematical Society \textbf{101} (2010), no.~1, 73--104.

\bibitem{Bourdaud:1991aa}
G.~Bourdaud, \emph{Le calcul fonctionnel dans les espaces de {S}obolev},
  Invent. Math. \textbf{104} (1991), no.~2, 435--446. 

\bibitem{Bronshtein79}
M.~D. Bronshtein, \emph{Smoothness of roots of polynomials depending on
  parameters}, Sibirsk. Mat. Zh. \textbf{20} (1979), no.~3, 493--501, 690,
  English transl. in Siberian Math. J. \textbf{20} (1980), 347--352.

\bibitem{ColombiniOrruPernazza12}
F.~Colombini, N.~Orr{{\`u}}, and L.~Pernazza, \emph{On the regularity of the
  roots of hyperbolic polynomials}, Israel J. Math. \textbf{191} (2012),
  923--944. 

\bibitem{Dudley:1977aa}
R.~M. Dudley, \emph{On second derivatives of convex functions}, Math. Scand.
  \textbf{41} (1977), no.~1, 159--174.

\bibitem{Garding51}
L.~G{\r{a}}rding, \emph{Linear hyperbolic partial differential equations with
  constant coefficients}, Acta Math. \textbf{85} (1951), 1--62.

\bibitem{Garding59}
\bysame, \emph{An inequality for hyperbolic polynomials}, J. Math. Mech.
  \textbf{8} (1959), 957--965.

\bibitem{Gurvits:2008aa}
L.~Gurvits, \emph{Van der {W}aerden/{S}chrijver-{V}aliant like conjectures and
  stable (aka hyperbolic) homogeneous polynomials: one theorem for all},
  Electron. J. Combin. \textbf{15} (2008), no.~1, Research Paper 66, 26, With a
  corrigendum. 

\bibitem{Hartman:1959aa}
P.~Hartman, \emph{On functions representable as a difference of convex
  functions}, Pacific J. Math. \textbf{9} (1959), no.~3, 707--713.

\bibitem{Harvey:2013aa}
F.~R. Harvey and H.~B. Lawson, Jr., \emph{G\aa rding's theory of hyperbolic
  polynomials}, Comm. Pure Appl. Math. \textbf{66} (2013), no.~7, 1102--1128.

\bibitem{Helton:2007aa}
J.~W. {Helton} and V.~{Vinnikov}, \emph{{Linear matrix inequality
  representation of sets}}, {Commun. Pure Appl. Math.} \textbf{60} (2007),
  no.~5, 654--674 (English).

\bibitem{Hiriart-Urruty:1985aa}
J.-B. Hiriart-Urruty, \emph{Generalized differentiability, duality and
  optimization for problems dealing with differences of convex functions},
  Convexity and duality in optimization ({G}roningen, 1984), Lecture Notes in
  Econom. and Math. Systems, vol. 256, Springer, Berlin, 1985, pp.~37--70.

\bibitem{KM03}
A.~Kriegl and P.~W. Michor, \emph{Differentiable perturbation of unbounded
  operators}, Math. Ann. \textbf{327} (2003), no.~1, 191--201.

\bibitem{KurdykaPaunescu08}
K.~Kurdyka and L.~Paunescu, \emph{Hyperbolic polynomials and multiparameter
  real-analytic perturbation theory}, Duke Math. J. \textbf{141} (2008), no.~1,
  123--149.

\bibitem{Lax:1958aa}
P.~D. Lax, \emph{Differential equations, difference equations and matrix
  theory}, Comm. Pure Appl. Math. \textbf{11} (1958), 175--194. 

\bibitem{Leoni09}
G.~Leoni, \emph{A first course in {S}obolev spaces}, Graduate Studies in
  Mathematics, vol. 105, American Mathematical Society, Providence, RI, 2009.

\bibitem{Lewis:2005aa}
A.~S. {Lewis}, P.~A. {Parrilo}, and M.~V. {Ramana}, \emph{{The Lax conjecture
  is true.}}, {Proc. Am. Math. Soc.} \textbf{133} (2005), no.~9, 2495--2499
  (English).

%\bibitem{Mandai85}
%T.~Mandai, \emph{Smoothness of roots of hyperbolic polynomials with respect to
%  one-dimensional parameter}, Bull. Fac. Gen. Ed. Gifu Univ. (1985), no.~21,
%  115--118.

\bibitem{Marcus:2015aa}
A.~W. Marcus, D.~A. Spielman, and N.~Srivastava, \emph{Interlacing families
  {II}: {M}ixed characteristic polynomials and the {K}adison-{S}inger problem},
  Ann. of Math. (2) \textbf{182} (2015), no.~1, 327--350. 

\bibitem{Nuij68}
W.~Nuij, \emph{A note on hyperbolic polynomials}, Math. Scand. \textbf{23}
  (1968), 69--72 (1969).

\bibitem{ParusinskiRainerHyp}
A.~Parusi{{\'n}}ski and A.~Rainer, \emph{A new proof of {B}ronshtein's
  theorem}, J. Hyperbolic Differ. Equ. \textbf{12} (2015), no.~4, 671--688.

\bibitem{RainerQA}
A.~Rainer, \emph{Quasianalytic multiparameter perturbation of polynomials and
  normal matrices}, Trans. Amer. Math. Soc. \textbf{363} (2011), no.~9,
  4945--4977.

\bibitem{RainerN}
\bysame, \emph{Perturbation theory for normal operators}, Trans. Amer. Math.
  Soc. \textbf{365} (2013), no.~10, 5545--5577. 

\bibitem{Rellich37I}
F.~Rellich, \emph{St\"orungstheorie der {S}pektralzerlegung}, Math. Ann.
  \textbf{113} (1937), no.~1, 600--619.

\bibitem{Rellich69}
\bysame, \emph{Perturbation theory of eigenvalue problems}, Assisted by J.
  Berkowitz. With a preface by Jacob T. Schwartz, Gordon and Breach Science
  Publishers, New York, 1969.

\bibitem{Rockafellar70}
R.~T. Rockafellar, \emph{Convex analysis}, Princeton Mathematical Series, No.
  28, Princeton University Press, Princeton, N.J., 1970.

\bibitem{Vesely:2003aa}
L.~{Vesel\'y}, J.~{Duda}, and L.~{Zaj\'{\i}\v{c}ek}, \emph{{On d.c.~functions
  and mappings.}}, {Atti Semin. Mat. Fis. Univ. Modena} \textbf{51} (2003),
  no.~1, 111--138 (English).

\bibitem{VeselyZajicek89}
L.~Vesel{{\'y}} and L.~Zaj{\'{\i}}{\v{c}}ek, \emph{Delta-convex mappings
  between {B}anach spaces and applications}, Dissertationes Math. (Rozprawy
  Mat.) \textbf{289} (1989), 52. 

\bibitem{Vinnikov:1993aa}
V.~Vinnikov, \emph{Self-adjoint determinantal representations of real plane
  curves}, Math. Ann. \textbf{296} (1993), no.~1, 453--479.

\end{thebibliography}
%\bibliographystyle{amsplain}

\def\cprime{$'$}
\providecommand{\bysame}{\leavevmode\hbox to3em{\hrulefill}\thinspace}
\providecommand{\MR}{\relax\ifhmode\unskip\space\fi MR }
\providecommand{\MRhref}[2]{%
  \href{http://www.ams.org/mathscinet-getitem?mr=#1}{#2}
}
\providecommand{\href}[2]{#2}

\end{document}